\documentclass[12pt]{article}

\usepackage{amsmath}
\usepackage{amsthm}
\usepackage{amsfonts}
\usepackage[title]{appendix}

\usepackage{graphicx}
\graphicspath{ {./images/} }

\usepackage{enumitem}
\setlist[enumerate]{label={\upshape(\alph*)}}

\usepackage[margin=1.2in]{geometry}

\usepackage{color}

\usepackage{hyperref}

\usepackage{tikz}
\usetikzlibrary{calc}
\usetikzlibrary{decorations.pathreplacing}
\tikzstyle{vertex}=[circle, draw, inner sep=0pt, minimum size=5pt,fill=black]
\newcommand{\vertex}{\node[vertex]}
\tikzstyle{redvertex}=[rectangle, draw, inner sep=0pt, minimum size=5pt,fill=red]
\newcommand{\redvertex}{\node[redvertex]}
\tikzstyle{bluevertex}=[circle, draw, inner sep=0pt, minimum size=5pt,fill=white]
\newcommand{\bluevertex}{\node[bluevertex]}

\newtheorem{theorem}{Theorem}[section]

\newtheorem{lemma}[theorem]{Lemma}
\newtheorem{corollary}[theorem]{Corollary}

\theoremstyle{definition}
\newtheorem{definition}[theorem]{Definition}
\newtheorem{conjecture}[theorem]{Conjecture}

\theoremstyle{remark}
\newtheorem{remark}[theorem]{Remark}

\begin{document}

\title{On the Average (Edge-)Connectivity of Minimally $k$-(Edge-)Connected Graphs}
\author{Lucas Mol\thanks{Supported by an NSERC Grant CANADA, Grant number RGPIN-2021-04084}  and Ortrud R. Oellermann\thanks{Supported by an NSERC Grant CANADA, Grant number RGPIN-2016-05237}\\
University of Winnipeg\\
515 Portage Ave.\ Winnipeg, MB, Canada R3B 2E9\\
\small \href{mailto:l.mol@uwinnipeg.ca}{l.mol@uwinnipeg.ca},
\href{mailto:o.oellermann@uwinnipeg.ca}{o.oellermann@uwinnipeg.ca}\\
Vibhav Oswal \thanks{Supported by MITACS Globalink Scholarship}\\
BITS Pilani \\
Pilani, Rajasthan, India 333031\\
\small \href{mailto:vibhavoswal3@gmail.com}{vibhavoswal3@gmail.com}}

\date{}
\maketitle

\begin{abstract}
Let $G$ be a graph of order $n$ and let $u,v$ be vertices of $G$. Let $\kappa_G(u,v)$ 
denote the maximum number of internally disjoint 
$u$--$v$ paths in $G$. Then the \emph{average connectivity} $\overline{\kappa}(G)$ of $G$,  is defined as
$
\overline{\kappa}(G)=\sum_{\{u,v\}\subseteq V(G)} \kappa_G(u,v)/\tbinom{n}{2}. $ If $k \ge 1$ is an integer, then  $G$ is {\em minimally $k$-connected}  if $\kappa(G)=k$ and $\kappa(G-e) < k$ for every edge $e$ of $G$. 
We say that $G$ is an \emph{optimal} minimally $k$-connected graph if $G$ has maximum average connectivity among all minimally $k$-connected graphs of order $n$. 
Based on a recent structure result for minimally 2-connected graphs we conjecture that, for every integer $k \ge3$, if $G$ is an optimal minimally $k$-connected graph of order $n\geq 2k+1$, then $G$
is bipartite, with the set of vertices of degree $k$ and the set of vertices of degree exceeding $k$ as its partite sets. We show that if this conjecture is true, then $\overline{\kappa}(G)< \frac{9}{8}k$ for every minimally $k$-connected graph $G$.
For every $k \ge 3$, we describe an infinite family of minimally $k$-connected graphs whose average connectivity is asymptotically $\frac{9}{8}k$. Analogous results are established for the average edge-connectivity of minimally $k$-edge-connected graphs.

\end{abstract}

\section{Introduction}\label{Intro}

Let $G$ be a nontrivial graph.  The {\em connectivity} of $G$, denoted  by $\kappa(G)$, is the smallest number of vertices whose removal disconnects $G$ or produces a trivial graph. The {\em edge-connectivity} of $G$, denoted by $\lambda(G)$, is the smallest number of edges whose removal disconnects $G$ or produces a trivial graph.

Following Beineke, Oellermann, and Pippert~\cite{BeinekeOellermannPippert2002}, for a pair $u,v$ of distinct vertices of $G$, we define the \emph{connectivity} between $u$ and $v$ in $G$, denoted by $\kappa_G(u,v)$, to be the maximum number of internally disjoint $u$--$v$ paths.  The \emph{edge-connectivity} between $u$ and $v$, denoted by $\lambda_G(u,v)$, is the maximum number of edge-disjoint $u$--$v$ paths. Menger's well-known theorem~\cite{Menger1927} states that if $u$ and $v$ are non-adjacent, then $\kappa_G(u,v)$ equals the smallest number of vertices whose removal from $G$ separates $u$ and $v$.  The edge-connectivity version of Menger's theorem states that $\lambda_G(u,v)$ equals the minimum number of edges whose removal from $G$ separates $u$ and $v$. When $G$ is clear from context we omit the subscript $G$ from $\kappa_G(u,v)$ and $\lambda_G(u,v)$.  

Whitney \cite{Whitney1932} showed that $\kappa(G) = \min\{\kappa(u,v)\ |\ u,v \in V(G)\}$. In a similar manner it follows that $\lambda(G) = \min\{\lambda(u,v)\ |\ u,v \in V(G)\}$.  These results show that both the connectivity and the edge-connectivity of a graph are worst-case measures. A more refined measure of the overall level of connectedness of a graph was introduced by Beineke, Oellermann, and Pippert~\cite{BeinekeOellermannPippert2002}, and is based on the average values of the `local connectivities' between all pairs of vertices.  The {\em average connectivity} of a graph $G$ of order $n$, denoted by $\overline{\kappa}(G)$, is the average of the connectivities over all pairs of distinct vertices of $G$.  That is,
\[
\overline{\kappa}(G)= \sum_{\{u,v \}\subseteq V(G)}\kappa(u,v)/\tbinom{n}{2}.
\]
Analogously, the {\em average edge-connectivity} of $G$, studied by Dankelmann and Oellermann~\cite{DankelmannOellermann2005}, and denoted by $\overline{\lambda}(G)$, is the average of the edge-connectivities over all pairs of distinct vertices of $G$. That is,
\[
\overline{\lambda}(G)= \sum_{\{u,v \}\subseteq V(G)}\lambda(u,v)/\tbinom{n}{2}.
\]

Several bounds for the average connectivity in terms of various graph parameters, such as for example, the order and size~\cite{BeinekeOellermannPippert2002}, the average degree~\cite{DankelmannOellermann2003}, and the matching number~\cite{KimO2013} have been determined.  Bounds  on the average connectivity of graphs belonging to particular families have also been established, including bounds for planar and outerplanar graphs~\cite{DankelmannOellermann2003}, Cartesian product graphs~\cite{DankelmannOellermann2003}, strong product graphs~\cite{Abajo2013}, and regular graphs~\cite{KimO2013}.  The average connectivity also plays a role in the assessment of the reliability of real-world networks, including street networks~\cite{Boeing2017} and communication networks~\cite{Rak2015}.

In this paper we study by how much the average (edge-)connectivity can vary in a class of graphs, whose members are in some sense just barely $k$-(edge-)connected for some integer $k \ge 1$.
A graph $G$ is called {\em minimally $k$-connected} if $\kappa(G)=k$ and $\kappa(G-e) < k$ for every edge $e$ of $G$.   {\em  Minimally $k$-edge-connected}  graphs are defined similarly.  It is natural to ask by how much the average (edge-)connectivity of a minimally $k$-(edge-)connected graph can differ from $k$. Trivially the smallest average (edge-)connectivity  among all minimally $k$-(edge-)connected graph is $k$. For the remainder of the paper we thus focus on an upper bound for the average connectivity for all minimally $k$-(edge-)connected graphs. We say that $G$ is an \emph{optimal} minimally $k$-connected graph if $G$ has maximum average connectivity among all minimally $k$-connected graphs. Since minimally 1-(edge-)connected graphs are precisely the trees, they have average connectivity $1$. However, for $k \ge 2$, the average (edge-)connectivity of a minimally $k$-(edge-)connected graph need not be $k$. The structure of optimal minimally $2$-(edge-)connected graphs, and an upper bound on their average (edge-) connectivity is determined by Casablanca, Mol, and Oellermann \cite{CasablancaMolOellermann2018}. In order to state these results we say that a minimally $k$-(edge-)connected graph is {\em degree-partitioned} if it is bipartite, with partite sets the set of vertices of degree $k$ and the set of vertices of degree exceeding $k$.   (Note that every degree-partitioned minimally $k$-(edge-)connected graph has order at least $2k+1$.)

\begin{theorem}[Casablanca, Mol, and Oellermann~\cite{CasablancaMolOellermann2018}]\  \label{min_2_conn}
\begin{enumerate}
\item If $G$ is an optimal minimally $2$-connected graph of order $n\geq 5$, then $G$ is degree-partitioned.  Moreover, we have $\overline{\kappa}(G)<\frac{9}{4}$, and this bound is asymptotically sharp.
\item If $G$ is an optimal minimally $2$-edge-connected graph of order $n\geq 5$, then $G$ is degree-partitioned.  Moreover, we have $\overline{\lambda}(G)<\frac{9}{4}$, and this bound is asymptotically sharp.
\end{enumerate}
\end{theorem}

In this paper, we continue the study of the average (edge-)connectivity of minimally $k$-(edge-)connected graphs, which was initiated by Casablanca, Mol, and Oellermann~\cite{CasablancaMolOellermann2018}.  Mader \cite{Mader1972} showed that the vertices of degree exceeding $k$ in a minimally $k$-connected graph induce a forest. Based on Theorem \ref{min_2_conn}, and some computational evidence, we believe that something similar can be said about the structure of optimal minimally $k$-(edge-)connected graphs for $k\ge 3$. 

\begin{conjecture}\label{DegreePartitionedConjecture}
Let $k\geq 3$.  If $G$ is an optimal minimally $k$-(edge-)connected graph of order $n\geq 2k+1$, then $G$ is degree-partitioned.
\end{conjecture}

In Section \ref{preliminaries}, we show that if $k \ge 2$ and $G$ is a degree-partitioned minimally $k$-connected graph of order $n$, then the average connectivity of $G$ satisfies
\begin{align}\label{vertex_version}
\overline{\kappa}(G)\leq k+\frac{k(n-2)^2}{8n(n-1)}<\frac{9}{8}k.
\end{align}
By a similar argument, it follows that if $k\ge 2$ and $G$ is a degree-partitioned minimally $k$-edge-connected graph of order $n$, then the average edge-connectivity of $G$ satisfies 
\begin{align} \label{edge_version_bound}
\overline{\lambda}(G)\leq k+\frac{k(n-2)^2}{8n(n-1)}<\frac{9}{8}k.
\end{align}
We note that, if Conjecture~\ref{DegreePartitionedConjecture}
 holds, then every minimally $k$-connected graph $G$ satisfies $\overline{\kappa}(G)<\frac{9}{8}k$, and every minimally $k$-edge-connected graph $G$ satisfies $\overline{\lambda}(G)<\frac{9}{8}k$. The inequalities given in (\ref{vertex_version}) and (\ref{edge_version_bound}) were established in \cite{CasablancaMolOellermann2018} for the case $k=2$ and it was remarked that these proofs could be extended to all $k \ge 3$.

In Section \ref{max_lambda_bar} we describe, for every $k \ge 3$, an infinite family of degree-partitioned minimally $k$-edge-connected graphs whose average edge-connectivity is asymptotically $\frac{9}{8}k$.  In Section \ref{max_kappa_bar} we describe, for every $k\ge 3$, an infinite family of degree-partitioned minimally $k$-connected graphs whose average connectivity is asymptotically $\frac{9}{8}k$.  Thus, the upper bounds given by (\ref{vertex_version}) and (\ref{edge_version_bound}) are asymptotically sharp.

\section{Upper Bounds}\label{preliminaries}

In order to establish the upper bounds given by (\ref{vertex_version}) and (\ref{edge_version_bound}), we generalize the argument given by Casablanca, Mol, and Oellermann~\cite[Section~2.2]{CasablancaMolOellermann2018} for $k=2$.  We first recall some terminology (c.f.~\cite{CasablancaMolOellermann2018}).

Let $G$ be a graph of order $n$.  The \emph{total connectivity} of $G$, denoted by $K(G)$, is the sum of the connectivities over all pairs of distinct vertices of $G$, i.e., we have $K(G)=\binom{n}{2}\overline{\kappa}(G)$.  The \emph{potential} of a sequence of positive integers $d_1,d_2,\dots,d_n$ is defined by
\[
P(d_1,d_2,\dots,d_n)=\sum_{1\leq i<j\leq n} \min\{d_i,d_j\}.
\]
If $G$ has vertices $v_1,v_2,\dots,v_n$, then the \emph{potential} of $G$, denoted by $P(G)$, is the potential of the degree sequence of $G$; that is,
\begin{align*}
P(G)&=P(\deg(v_1),\deg(v_2),\dots,\deg(v_n))=\sum_{1\leq i<j\leq n}\min\{\deg(v_i),\deg(v_j)\}.
\end{align*}
Since $\kappa(u,v)\leq \min\{\deg(u),\deg(v)\}$ for all pairs of distinct vertices $u,v$ of $G$, we have $K(G)\leq P(G)$.

We require the following lemma, which describes the maximum potential among all sequences of $n$ positive integers whose sum is a fixed number $D$.

\begin{lemma}[Beineke, Oellermann, and Pippert~\cite{BeinekeOellermannPippert2002}]\label{Balanced}
Let $d_1,d_2,\dots,d_n$ be the degree sequence of a graph, and let $D=\sum_{i=1}^n d_i$.  Let $D=dn+r,$ where $d\geq 0$ and $0\leq r<n.$  Then
\[
P(d_1,d_2,\dots,d_n)\leq P(\underbrace{d,\dots,d}_{\text{$n-r$ terms}},\underbrace{d+1,\dots,d+1}_{\text{$r$ terms}}).  \qedhere
\]
\end{lemma}

We are now ready to prove the upper bound given by (\ref{vertex_version}).  Recall that a minimally $k$-connected graph is called degree-partitioned if it is bipartite, with partite sets the set of vertices of degree $k$ and the set of vertices of degree exceeding $k$.

\begin{theorem}\label{KappaBarBound}
Let $k\geq 2$, and let $G$ be a degree-partitioned minimally $k$-connected graph of order $n \ge 2k+1$.  Then
\[
\overline{\kappa}(G)\leq k+\frac{k(n-2)^2}{8n(n-1)} < \frac{9}{8}k.
\]
\end{theorem}

\begin{proof}
Suppose that $G$ has $s$ vertices of degree exceeding $k$, and hence $n-s$ vertices of degree $k$. Let $d_1, d_2, \ldots, d_s$ be the degrees of the vertices of degree exceeding $k$.  Since $G$ is degree-partitioned, the sum $d_1+d_2+\cdots+d_s$ must be equal to $k(n-s)$, the sum of the degrees of the vertices having degree $k$.

Let $k(n-s)=ds+r$ for $d,r\in\mathbb{Z}$ and $0\leq r<s$.  Then by Lemma~\ref{Balanced}, we have
\begin{align*}
K(G)\leq P(G)&\leq k\left[\tbinom{n}{2}-\tbinom{s}{2}\right]+P(d_1,d_2,\dots,d_s)\\
&\leq k\left[\tbinom{n}{2}-\tbinom{s}{2}\right]+P(\underbrace{d,\dots,d}_{\text{$s-r$ terms}},\underbrace{d+1,\dots,d+1}_{\text{$r$ terms}})\\
&\leq k\tbinom{n}{2}-k\tbinom{s}{2}+d\tbinom{s}{2}+\tbinom{r}{2}\\
&=k\tbinom{n}{2}+(d-k)\tbinom{s}{2}+\tbinom{r}{2}\\
&=k\tbinom{n}{2}+\left[\tfrac{k(n-s)-r}{s}-k\right]\tbinom{s}{2}+\tbinom{r}{2}\\
&=k\tbinom{n}{2}+\left[\tfrac{kn-2ks-r}{s}\right]\tfrac{s(s-1)}{2}+\tfrac{r(r-1)}{2}\\
&=k\tbinom{n}{2}+\tfrac{(kn-2ks)(s-1)}{2}-\tfrac{r(s-1)}{2}+\tfrac{r(r-1)}{2}\\
&=k\tbinom{n}{2}+\tfrac{k(n-2s)(s-1)}{2}-\tfrac{r(s-r)}{2}\\
&\leq k\tbinom{n}{2}+\tfrac{k}{2}(n-2s)(s-1)
\end{align*}
Using elementary calculus, we find that the quantity $(n-2s)(s-1)$ achieves a maximum of $\tfrac{(n-2)^2}{8}$ at $s=\tfrac{n+2}{4}$. Thus  we have
\begin{align*}
K(G)&\leq k\tbinom{n}{2}+\tfrac{k}{2}\tfrac{(n-2)^2}{8}.
\end{align*}
Now dividing through by $\binom{n}{2}$ gives the desired upper bound on $\overline{\kappa}(G)$.
\end{proof}

The upper bound given by~(\ref{edge_version_bound}) can be established in a strictly analogous manner, so we omit the proof.

\begin{theorem}\label{LambdaBarBound}
Let $k\geq 2$, and let $G$ be a degree-partitioned minimally $k$-edge-connected graph of order $n\geq 2k+1$.  Then
\[
\overline{\lambda}(G)\leq k+\frac{k(n-2)^2}{8n(n-1)} < \frac{9}{8}k. 
\]

\end{theorem}

\section{Constructions} \label{max_ave}

In this section, we provide constructions of degree-partitioned minimally $k$-connected graphs and degree-partitioned minimally $k$-edge-connected graphs for which the upper bounds of Theorem~\ref{KappaBarBound} and Theorem~\ref{LambdaBarBound}, respectively, are attained asymptotically.  This has already been done for the case $k=2$~\cite{CasablancaMolOellermann2018}, so we consider only $k\geq 3$.  We begin by defining a $k$-regular graph $G_{k,p}$, which is used as a ``building block'' in the constructions that follow.

\begin{definition} \label{G_{k,p}}
Let $k, p$ be integers such that $3 \le k \le p$. Let $W=\{w_0,w_1, \ldots, w_{p-1}\}$ and $X=\{x_0, x_1, \ldots, x_{p-1}\}$. Let $G_{k,p}$ be the graph with vertex set $W\cup X$ and edge set 

\[ E=\{w_ix_{i+j}\ |\ 0 \le i \le p-1, 0 \le j \le k-1\},\]
where subscripts are expressed modulo $p$.
\end{definition}

For example, the graph $G_{3,20}$ is illustrated in Figure~\ref{Fig:G3}.  In the sequel, the notation $G_{k,p}$ will always denote the graph of Definition \ref{G_{k,p}}.  We will require several intermediate results about $G_{k,p}$.  We begin with a lemma about the minimal vertex separators in $G_{k,p}$.  

For a vertex $u$ in a graph $G$, we let $N(u)$ denote the \emph{(open) neighbourhood} of $u$, i.e., the set of vertices adjacent to $u$ in $G$.  For a subset $U\subseteq V(G)$, we let $N(U)$ denote the \emph{(open) neighbourhood} of $U$, i.e., the set of vertices not contained in $U$ that are adjacent to a vertex in $U$.

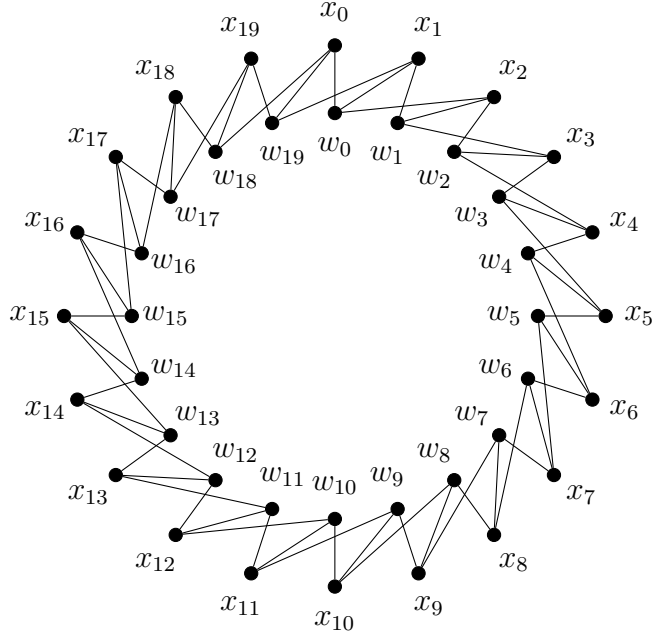
\begin{figure}
\begin{center}
\begin{tikzpicture}[scale=0.9]
\pgfmathtruncatemacro{\n}{20}
\pgfmathtruncatemacro{\m}{\n-1}
\foreach \x in {0,1,...,\m}
{
\vertex (\x) at (-\x/\n*360+90:3) {};
\node at (-\x/\n*360+90:2.5) {$w_{\x}$};
\vertex (a\x) at (-\x/\n*360+90:4) {};
\node at (-\x/\n*360+90:4.5) {$x_{\x}$};
}
\foreach \x in {0,1,2,...,\m}
{
\pgfmathtruncatemacro{\y}{Mod(\x-1,\n)}
\pgfmathtruncatemacro{\z}{Mod(\x-2,\n)}
\path
(\x) edge (a\x)
(\y) edge (a\x)
(\z) edge (a\x);
}
\end{tikzpicture}
\end{center}
\caption{The graph $G_{3,20}$.}
\label{Fig:G3}
\end{figure}

\begin{lemma}\label{non_triv_components}
Let $k,p$ be integers such that $3 \le k \le p$, and let $u$ and $v$ be nonadjacent vertices of $G_{k,p}$.  Let $S$ be a minimal vertex separator of $u$ and $v$ in $G_{k,p}$.  Then $|S|=k$ or $|S|=2k-2$.
\end{lemma}
\begin{proof}
First of all, if either $u$ or $v$ is isolated in $G_{k,p}-S$, say $u$, then $S$ contains the entire neighbourhood $N(u)$ of $u$, and by the minimality of $S$, we have $S=N(u)$.  We conclude that $|S|=k$ in this case.

So we may assume that neither $u$ nor $v$ is isolated in $G_{k,p}-S$.  In this case, we show that $|S|=2k-2$.  Let $C$ be the component of $G_{k,p}-S$ that contains $u$, and let $D$ be the union of the remaining components of $G_{k,p}-S$. Colour the vertices of $C$ red, the vertices of $D$ white, and the vertices of $S$ black.  Since $u$ is not isolated in $G_{k,p}-S$, the component $C$ has order at least $2$, and hence both $W$ and $X$ must contain at least one red vertex.  Similarly, since $v$ is not isolated in $G_{k,p}-S$, we see that both $W$ and $X$ must contain at least one white vertex.

By symmetry, we may assume that $w_0$
is red, and that 
$w_{p-1}$ is not red; otherwise, we can relabel the vertices of $G_{k,p}$ so that this happens. Since $S$ is a minimal vertex separator of $u$ and $v$, there are no edges between red and white vertices, and every black vertex must be adjacent with at least one red vertex and at least one white vertex.  We illustrate the relevant portion of the graph $G_{k,p}$ in Figure~\ref{Fig:S}.

Let $t\geq 0$ be the largest integer such that all of the vertices in the set $C_W=\{w_0, \ldots, w_t\}$ are coloured red.  Thus all of the vertices in the set $N(\{w_0, \ldots, w_t\}) = \{x_0, \ldots, x_{t+k-1}\}$ are coloured either red or black.  
Let $x_{\ell}$ be the first red vertex and $x_{r}$ be the last red vertex in the sequence $x_0, \ldots, x_{t+k-1}$. (We use $\ell$ and $r$ for ``left'' and ``right'', respectively.)  Since $C$ is connected and contains $w_0$, some neighbour of $w_0$ must be coloured red, meaning that $\ell\leq k-1$.  Similarly, some neighbour of $w_t$ must be coloured red, meaning that $r\geq t$.  So we have $0\leq \ell\leq k-1$ and $t\leq r\leq t+k-1$.

We show first that $x_j$ is coloured red for every $\ell < j < r$. Suppose otherwise that this is not the case, and let $j$ be the smallest integer such that $\ell<j<r$ and $x_j$ is coloured black.  Note that the black vertex $x_j$ must have a white neighbour.  By the minimality of $j$, the vertices $x_\ell,\ldots, x_{j-1}$ are coloured red, and hence none of the vertices $w_0,\ldots,w_{j-1}$ are coloured white.  Thus we either have $j>t$ and $w_j$ is white, or $j<k-1$ and $x_j$ has a white neighbour in the set $\{w_{p-(k-1-j)},\ldots,w_{p-1}\}$.  In the first case, the white vertex $w_j$ is also adjacent to the red vertex $x_r$, a contradiction.  In the second case, the white neighbour of $x_j$ is also adjacent to the red vertex $x_{\ell}$, a contradiction.  We have shown that $C$ contains the vertices in the set $C_W=\{w_0,\ldots, w_t\}$ and the vertices in the set $C_X=\{x_\ell,\ldots,x_r\}$.  In fact, we will see that $V(C)=C_W\cup C_X$.

We now show that $S$ has at least $2k-2$ vertices, i.e., that at least $2k-2$ vertices are coloured black.  First of all, by the definition of $\ell$ and $r$, and the fact that each of the vertices $x_0,\ldots,x_{t+k-1}$ is either red or black, we see that the vertices in the sets
\[
L_X=\{x_0,\ldots,x_{\ell-1}\} \ \ \text{ and } \ \
R_X=\{x_{r+1},\ldots,x_{t+k-1}\}
\]
are coloured black.  (Note that the set $L_X$ is empty if $\ell=0$, and that the set $R_X$ is empty if $r=t+k-1$.)  

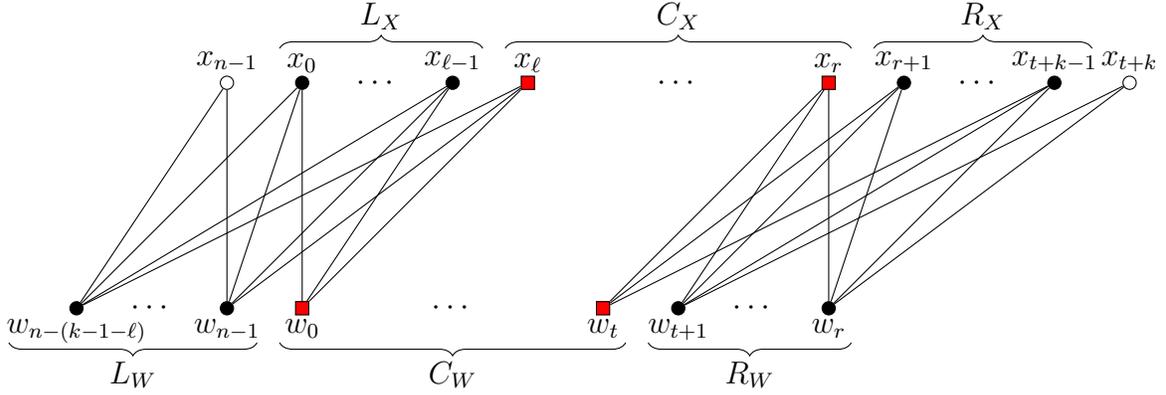
\begin{figure}
\begin{center}
\begin{tikzpicture}
\foreach \x in {0,2,8,10}
{
\vertex (\x) at (\x,0) {};
}
\foreach \x in {3,7}
{
\redvertex (\x) at (\x,0) {};
}
\foreach \x in{1,5,9}
{
\node at (\x,0) {$\cdots$};
}
\foreach \x in {3,5,11,13}
{
\vertex (-\x) at (\x,3) {};
}
\foreach \x in {6,10}
{
\redvertex (-\x) at (\x,3) {};
}
\foreach \x in{4,8,12}
{
\node at (\x,3) {$\cdots$};
}
\bluevertex(-14) at (14,3) {};
\bluevertex(-2) at (2,3) {};
\node[below] at (0) {$w_{n-(k-1-\ell)}$};
\node[below] at (2) {$w_{n-1}$};
\node[below] at (3) {$w_0$};
\node[below] at (7) {$w_t$};
\node[below] at (8) {$w_{t+1}$};
\node[below] at (10) {$w_r$};
\node[above] at (-2) {$x_{n-1}$};
\node[above] at (-3) {$x_{0}$};
\node[above] at (-5) {$x_{\ell-1}$};
\node[above] at (-6) {$x_{\ell}$};
\node[above] at (-10) {$x_{r}$};
\node[above] at (-11) {$x_{r+1}$};
\node[above] at (-13) {$x_{t+k-1}$};
\node[above] at (-14) {$x_{t+k}$};
\foreach \x in {0,2}
{
\foreach \y in {-2,-3,-5,-6}
{
\path
(\x) edge (\y);
}
}
\foreach \y in {-3,-5,-6}
{
\path (3) edge (\y);
}
\foreach \x in {8,10}
{
\foreach \y in {-10,-11,-13,-14}
{
\path
(\x) edge (\y);
}
}
\foreach \y in {-10,-11,-13}
{
\path (7) edge (\y);
}
\draw [decorate,decoration={mirror,brace,amplitude=5pt},xshift=0pt,yshift=-13pt]
(-0.9,0) -- (2.4,0) node [below,black,midway,yshift=-3pt]
{$L_W$};
\draw [decorate,decoration={mirror,brace,amplitude=5pt},xshift=0pt,yshift=-13pt]
(2.7,0) -- (7.3,0) node [below,black,midway,yshift=-3pt]
{$C_W$};
\draw [decorate,decoration={mirror,brace,amplitude=5pt},xshift=0pt,yshift=-13pt]
(7.6,0) -- (10.3,0) node [below,black,midway,yshift=-3pt]
{$R_W$};
\draw [decorate,decoration={brace,amplitude=5pt},xshift=0pt,yshift=13pt]
(2.7,3) -- (5.4,3) node [above,black,midway,yshift=3pt]
{$L_X$};
\draw [decorate,decoration={brace,amplitude=5pt},xshift=0pt,yshift=13pt]
(5.7,3) -- (10.3,3) node [above,black,midway,yshift=3pt]
{$C_X$};
\draw [decorate,decoration={brace,amplitude=5pt},xshift=0pt,yshift=13pt]
(10.6,3) -- (13.5,3) node [above,black,midway,yshift=3pt]
{$R_X$};
\end{tikzpicture}
\end{center}
\caption{The minimal vertex separator $S$ for $u$ and $v$. Note that the black and white vertices are represented by black and white circles, respectively, and the red vertices are represented by red squares.}
\label{Fig:S}
\end{figure}

We claim that the vertices in the sets 
\[
L_W=\{w_{p-(k-1-\ell)},\ldots, w_{p-1}\} \ \ \text{ and } \ \  R_W=\{w_{t+1},\ldots,w_r\}
\]
are also coloured black.  First consider the set $L_W$.  If $\ell=k-1$, then the set $L_W$ is empty, and there is nothing to prove.  So suppose $\ell<k-1$.  Then the vertex $w_{p-1}$ is adjacent to the red vertex $x_\ell$, and since we have assumed that $w_{p-1}$ is not red, it must be black.  Since every black vertex must have a white neighbour, and the neighbours $x_0,\ldots, x_{k-2}$ of $w_{p-1}$ are all black or red, the vertex $x_{p-1}$ must be coloured white.  So all of the vertices in $L_W$ are adjacent to the white vertex $x_{p-1}$ and the red vertex $x_{\ell}$, and must therefore be black.  The argument for $R_W$ is similar.  If $r=t$, then $R_W$ is empty, so suppose that $r>t$.  By the maximality of $t$, the vertex $w_{t+1}$ must be black, and hence must have a white neighbour.  It follows that the vertex $x_{t+k}$ must be white.  So all of the vertices in $R_W$ are adjacent to the white vertex $x_{t+k}$ and the red vertex $x_r$, and must therefore be black.

Let $T=L_W\cup L_X\cup R_W\cup R_X$.  We have shown that $T\subseteq S$.  Since both $W$ and $X$ contain white vertices, the sets $L_X$, $R_X$, $L_W$, and $R_W$ are pairwise disjoint.  Note also that $|L_W\cup L_X|=k-1$ and $|R_W\cup R_X|=k-1$, so $|S|\geq |T|= 2k-2$.  Moreover, since $N(C_W\cup C_X)=T$, we see that no vertex in $C_W\cup C_X$ has a red neighbour outside of $C_W\cup C_X$.  It follows that $V(C)=C_W\cup C_X$, and that $G_{k,p}-T$ is disconnected, hence $S=T$.
\end{proof}

We note that in the terminology of \cite{LiYang2012}, we have shown that $G_{p,k}$ has connectivity $k$ and is \emph{essentially $(2k-2)$-connected}.

The proof of Lemma~\ref{non_triv_components} reveals more about the minimal vertex separators of $G_{k,p}$ than just their cardinality.  We can describe the structure of the minimal vertex separators in $G_{p,k}$ as in the following remark and this is important in the sequel.

\begin{remark}\label{Remark:VertexSeparators}
Let $S$ be a minimal vertex separator of nonadjacent vertices $u$ and $v$ in $G_{k,p}$.  Then one of the following holds:
\begin{itemize}
\item $|S|=k$, and $S=N(u)$ or $S=N(v)$.
\item $|S|=2k-2$, and with notation as in the proof of Lemma~\ref{non_triv_components}, we have $S\cap W=L_W\cup R_W$,  where both $L_W$ and $R_W$ consist of at most $k-1$ consecutive vertices from the cyclic arrangement of vertices of $W$ and $S\cap X=L_X\cup R_X$, where $L_X$ and $R_X$ consist of at most $k-1$ consecutive vertices from the cyclic arrangement of vertices of $X$. Moreover, $|L_W \cup L_X| =k-1$ and $|R_W \cup R_X|=k-1$.
\end{itemize}
\end{remark}

It is also straightforward to prove that $G_{k,p}$ is minimally $k$-connected and minimally $k$-edge-connected using Lemma~\ref{non_triv_components}.

\begin{corollary} \label{minkcon}
Let $k,p$ be integers such that $3 \le k \le p$. Then $G_{k,p}$ is minimally $k$-connected and minimally $k$-edge-connected.
\end{corollary}

\begin{proof} 
Since $G_{k,p}$ is $k$-regular, we must have $\kappa(G_{k,p})\leq \lambda(G_{k,p})\leq k$.  Now let $S$ be a minimal vertex separator of $G_{k,p}$.  By Lemma~\ref{non_triv_components}, we have $|S|=k$ or $|S|=2k-2$.  Since $k\geq 3$, we have $2k-2>k$, hence $|S|\geq k$.  So $\kappa(G_{k,p})\geq k$, and we conclude that $\kappa(G_{k,p})=\lambda(G_{k,p})=k$.  Finally, since $G_{k,p}$ is $k$-regular, we see that for every edge $e$ of $G_{k,p}$, we have $\kappa(G_{k,p}-e)<k$ and $\lambda(G_{k,p}-e)<k$.  Thus, we conclude that $G_{k,p}$ is minimally $k$-connected and minimally $k$-edge-connected.
\end{proof}

\subsection{Minimally \texorpdfstring{\boldmath$k$}{k}-edge connected graphs}\label{max_lambda_bar}

We show in this subsection that for all $k\geq 3$, there is an infinite family of degree-partitioned minimally $k$-edge-connected graphs whose average edge-connectivity asymptotically achieves the $\frac{9}{8}k$ upper bound established in Section~\ref{preliminaries}.

\begin{theorem} \label{sharpness_9/8_bound_for_edges}
There is an infinite family of degree-partitioned minimally $k$-edge-connected graphs whose average edge-connectivity is asymptotically $\frac{9}{8}k$.
\end{theorem}

\begin{proof}
Let $k,p$ be integers such that $3 \le k \le p$.  Let $W=\{w_0,w_1,\ldots, w_{p-1}\}$, and for $m \in \{1,2,3\}$, let $X_m=\{x^{m}_{0}, x^{m}_{1}, \ldots, x^{m}_{p-1}\}$.  Let $\Gamma_{k,p}$ be the graph of order $4p$ with vertex set $W\cup X_1\cup X_2\cup X_3$ and edge set $E_1\cup E_2\cup E_3$, where
\[ 
E_m=\{w_ix^{m}_{i+j}\ |\ 0 \le i \le p-1, 0 \le j \le k-1\}
\]
for $m \in \{1,2,3\}$, and where subscripts are expressed modulo $p$.  For $m\in\{1,2,3\}$, let $H_m=\Gamma_{k,p}[W\cup X_m]$.  Note that $H_m\cong G_{k,p}$ for all $m\in\{1,2,3\}$, and that $H_1$, $H_2$, and $H_3$ are pairwise edge-disjoint.  The graph $\Gamma_{k,p}$ is bipartite with partite sets $W$ and $X=X_1\cup X_2\cup X_3$, and every vertex in $W$ has degree $3k$, while every vertex in $X$ has degree $k$.  (Essentially, the graph $\Gamma_{k,p}$ consists of three copies of $G_{k,p}$, where the three copies of the vertex $w_i$ are identified for all $0\leq i<p$.) By Corollary~\ref{minkcon}, the graph $G_{k,p}$ is $k$-edge-connected, so it follows that $\Gamma_{k,p}$ is $k$-edge-connected.  Further, since every edge of $\Gamma_{k,p}$ is incident to a vertex of degree $k$, we see that $\Gamma_{k,p}$ is minimally $k$-edge-connected.

We now compute the average connectivity of $\Gamma_{k,p}$.  First of all, if $x \in X$ and $v \in V(\Gamma_{k,p})-\{x\}$, then $\lambda(x,v) = k$, since $\Gamma_{k,p}$ is $k$-edge-connected and $\deg (x) =k$.  If $w_i,w_j \in W$ for $i \ne j$, then $\lambda(w_i,w_j) = 3k$, since there are $k$ edge-disjoint $w_i$--$w_j$ paths in each of the edge-disjoint subgraphs $H_1$, $H_2$ and $H_3$.  Thus the average edge-connectivity of $\Gamma_{k,p}$ is given by $$\frac{3k\binom{p}{2} + k [\binom{4p}{2} -\binom{p}{2}]}{\binom{4p}{2}}=\left(\frac{9p-3}{8p-2}\right)k,$$ which is asymptotically $\frac{9}{8}k$.
\end{proof}

\subsection{Minimally \texorpdfstring{\boldmath$k$}{k}-connected graphs}\label{max_kappa_bar}

We show in this subsection that for all $k\geq 3$, there is an infinite family of degree-partitioned minimally $k$-connected graphs whose average connectivity asymptotically achieves the $\frac{9}{8}k$ upper bound established in Section~\ref{preliminaries}.  

Before we proceed, we explain why we need a different construction than the graph $\Gamma_{k,p}$ described in the proof of Theorem~\ref{sharpness_9/8_bound_for_edges}.  After all, it is not hard to see that the graph $\Gamma_{k,p}$ is minimally $k$-connected.  But while the average edge-connectivity of $\Gamma_{k,p}$ is asymptotically $\frac{9}{8}k$, the average connectivity of $\Gamma_{k,p}$ is not.   For suppose that $k\geq 3$ and $p\geq 2k$.  Then $\{w_1,w_2, \ldots, w_{k-1}\} \cup \{w_{p-k+1}, w_{p-k+2}, \ldots, w_{p-1}\}$ is a vertex separator for $w_0$ and $w_j$ for all $j\in\{k,k+1,\ldots, p-k\}$. By symmetry, we have $\kappa(w_i,w_{i+j}) \le 2k-2< 3k$ for all $i\in\{0,1,\ldots, p-1\}$ and all $j\in\{k,k+1,\ldots, p-k\}$.  It follows that the average connectivity of $\Gamma_{k,p}$ is not asymptotically $\frac{9}{8}k$.

So we define different families of degree-partitioned minimally $k$-connected graphs for which the upper bound given in Theorem~\ref{KappaBarBound} is attained asymptotically.  We require two slightly different constructions; one for $k\in\{3,4,5\}$, where we compute the average connectivity by constructing internally disjoint paths, and another for $k\geq 6$, where we compute the average connectivity by considering vertex separators.  We will explain why neither of these two approaches can be easily adapted to work for all $k\geq 3$.  We begin by considering $k\in\{3,4,5\}$.

\begin{theorem}\label{upperbound_sharp}
If $k\geq 3$, then there is an infinite family of degree-partitioned minimally $k$-connected graphs whose average connectivity is asymptotically $\frac{9}{8}k$.
\end{theorem}

\begin{proof} For $k \in \{3,4,5\}$, the proof is completed using a computer algebra system. The details are omitted here but are included in Appendix \ref{cases345}.

Assume now that $k\geq 6$ is fixed and let $p \in \{rk^2-1\ |\ r \in \{k+1, k+2, \ldots\}\}$. Let $\pi_1$ and $\pi_2$ be functions from the set $\{0,1,2, \ldots, p-1\}$ to itself defined by 
\[\pi_1(i) = ki\]
and 
\[\pi_2(i) = k^2i\]
 for $0 \le i \le p-1$, and where the output in either case is expressed modulo $p$. By our choice of $p$, we see that the sequence of integers $0,k, 2k, \ldots ,(p-1)k$ modulo $p$ produced by $\pi_1$ contains $p$ distinct elements, and the sequence $0,k^2, 2k^2, \ldots, (p-1)k^2$ of integers modulo $p$ produced by $\pi_2$ also contains $p$ distinct elements. So $\pi_1$ and $\pi_2$ are permutations of the set $\{0,1, \ldots, p-1\}$. 

Let
\begin{align*}
W&=\{w_0,w_1,\ldots,w_{p-1}\},\\
X&=\{x_0,x_1,\ldots,x_{p-1}\},\\
Y&=\{y_0,y_1,\ldots,y_{p-1}\}, \text{ and}\\
Z&=\{z_0,z_1,\ldots,z_{p-1}\}.
\end{align*}
Let $\Phi_{k,p}$ be the graph of order $4p$ with vertex set $W\cup X\cup Y\cup Z$ and edge set $E_X\cup E_Y\cup E_Z$, where
\begin{align*}
E_X&=\{w_ix_{i+j}\ |\ 0 \le i \le p-1, 0\leq j\leq k-1\},\\
E_Y&=\{w_{\pi_1(i)}y_{i+j}\ |\ 0 \le i \le p-1, 0\leq j\leq k-1\}, \text{ and}\\
E_Z&=\{w_{\pi_2(i)}z_{i+j}\ |\ 0 \le i \le p-1, 0\leq j\leq k-1\},
\end{align*}
where subscripts are expressed modulo $p$.  

Let 
\begin{align*}
H_X&=\Phi_{k,p}[W\cup X],\\
H_Y&=\Phi_{k,p}[W\cup Y], \text{ and}\\
H_Z&=\Phi_{k,p}[W\cup Z].
\end{align*}
Note that $H_X$, $H_Y$, and $H_Z$ are isomorphic to $G_{k,p}$.  
 
 We first show that $\Phi_{k,p}$ is degree-partitioned minimally $k$-connected.  Let $S$ be any subset of at most $k-1$ vertices of $\Phi_{k,p}$. We will show that $\Phi_{k,p}-S$ is connected.  Since $H_X$ is isomorphic to $G_{k,p}$, it is $k$-connected by Corollary~\ref{minkcon}.  Therefore, the graph $H_X-S$ is connected.  Let $v$ be any vertex in $Y\cup Z$ that is not in $S$.  Then $v$ has $k$ neighbours in $\Phi_{k,p}$, all of which belong to $W\subseteq V(H_X)$.  At most $k-1$ of these neighbours belong to $S$, so $v$ is joined to some vertex of $H_X-S$.  It follows that $\Phi_{k,p}-S$ is connected, and hence $\Phi_{k,p}$ is $k$-connected.  Note that $\Phi_{k,p}$ is bipartite with partite sets $W$ and $X\cup Y\cup Z$, and that every vertex in $W$ has degree $3k$, while every vertex in $X\cup Y\cup Z$ has degree $k$.  We conclude that $\Phi_{k,p}$ is degree-partitioned minimally $k$-connected.


It remains to be shown that $\kappa(u,v)=3k$ for every pair of distinct vertices $u,v\in W$.  Since $u$ and $v$ both have degree $3k$, we certainly have $\kappa(u,v)\leq 3k$.  So it suffices to show that $|S|\geq 3k$ for every vertex separator $S$ of $u$ and $v$.  Let $S$ be a vertex separator of $u$ and $v$, and let $S_X$, $S_Y$, and $S_Z$ denote the sets $S\cap V(H_X)$, $S\cap V(H_Y)$, and $S\cap V(H_Z)$, respectively.  Note that $S_X$, $S_Y$ and $S_Z$ separate $u$ and $v$ in $H_X$, $H_Y$, and $H_Z$, respectively.  Let $T_X\subseteq S_X$, $T_Y\subseteq S_Y$ and $T_Z\subseteq S_Z$ be minimal separators of $u$ and $v$ in $H_X$, $H_Y$ and $H_Z$, respectively.  Note that we have 
\[
|S|=|S_X \cup S_Y \cup S_Z|\geq |T_X \cup T_Y \cup T_Z|.
\]
We will use the principle of inclusion and exclusion to show that $|T_X\cup T_Y\cup T_Z|\geq 3k$.

Let $\mathcal{T}=\{T_X,T_Y,T_Z\}$, and let $T\in\mathcal{T}$.  First of all, by Lemma~\ref{non_triv_components}, we have $|T|=k$ or $|T|=2k-2$.  Since $k\geq 6$, we have $2k-2>k$, so $|T|\geq k$.  Further, by Remark~\ref{Remark:VertexSeparators}, if $|T|=k$, then $T$ is the neighbourhood of $u$ or $v$ in the subgraph $H_X$, $H_Y$, or $H_Z$ corresponding to $T$, and since $u,v\in W$, we see that $T\cap W=\emptyset$ in this case.

We show now that if two distinct sets in $\mathcal{T}$ have nonempty intersection, then they both have cardinality $2k-2$, and their intersection has cardinality at most four.  Suppose first that $T_X\cap T_Y\neq \emptyset$.  Since $T_X\subseteq W\cup X$ and $T_Y\subseteq W\cup Y$, we see that $T_X\cap T_Y\subseteq W$.  Thus, from the previous paragraph, we must have $|T_X|=|T_Y|=2k-2$.  Further, by Remark~\ref{Remark:VertexSeparators}, we have 
\[
T_X\cap W\subseteq \{w_{a},w_{a+1},\ldots,w_{a+k-2}\}\cup \{w_{b},w_{b+1},\ldots,w_{b+k-2}\}\]
for some $a,b\in\{0,1,\ldots,p-1\}$, and
\begin{align*}
T_Y\cap W&\subseteq \{w_{\pi_1(c)},w_{\pi_1(c+1)},\ldots, w_{\pi_1(c+k-2)}\}\cup \{w_{\pi_1(d)},w_{\pi_1(d+1)},\ldots, w_{\pi_1(d+k-2)}\}\\
&=\{w_{kc},w_{kc+k},\ldots, w_{kc+k(k-2)}\}\cup \{w_{kd},w_{kd+k},\ldots, w_{kd+k(k-2)}\}
\end{align*}
for some $c,d\in\{0,1,\ldots,p-1\}$.  Since each of the sets $\{w_{a},w_{a+1},\ldots,w_{a+k-2}\}$ and $\{w_{b},w_{b+1},\ldots,w_{b+k-2}\}$ overlaps with each of the sets $\{w_{kc},w_{kc+k},\ldots, w_{kc+k(k-2)}\}$ and $\{w_{kd},w_{kd+k},\ldots, w_{kd+k(k-2)}\}$ in at most one vertex, we have $|T_X\cap T_Y|\leq 4$.  The arguments for $T_X\cap T_Z$ and $T_Y\cap T_Z$ are similar, and are omitted.

We now show that $|T_X\cup T_Y\cup T_Z|\geq 3k$ by considering several cases.
\begin{itemize}
\item If the sets in $\mathcal{T}$ are pairwise disjoint, then they each have cardinality at least $k$, and it follows immediately that $|T_X\cup T_Y\cup T_Z|\ge 3k$.
\item If exactly one pair of sets from $\mathcal{T}$ has nonempty intersection, then both of these sets have cardinality $2k-2$, and they overlap in at most four vertices.  Further, they are disjoint from the third set, which has cardinality at least $k$.  Thus, by the principle of inclusion and exclusion, we have 
\[
|T_X \cup T_Y \cup T_Z| \ge 2(2k-2)-4+k=5k-8>3k,  
\]
where we used the fact that $k\geq 6$ at the end.
\item If all pairs of sets in $\mathcal{T}$ have nonempty intersection, then all of the sets in $\mathcal{T}$ have cardinality $2k-2$, and each pair overlaps in at most four vertices.  Thus, by the principle of inclusion and exclusion, we have 
\[
|T_X\cup T_Y\cup T_Z|\ge 3(2k-2)-3(4)=6k-18\geq 3k,
\]
where we used the fact that $k\geq 6$ at the end.
\end{itemize}
We conclude in all cases that $|S|\geq |T_X\cup T_Y\cup T_Z|\geq 3k$. Therefore, we have $\kappa(u,v)=3k$, which completes the proof.
\end{proof}

\section{Conclusion}

The obvious open problem is to resolve Conjecture~\ref{DegreePartitionedConjecture}, which states that if $G$ is an optimal minimally $k$-(edge-)connected graph of order $n\geq 2k+1$ for some $k \ge 3$, then $G$ is degree-partitioned.  We showed that if this conjecture is true, then the average (edge-)connectivity of a minimally $k$-(edge-)connected graph is at most $\frac{9}{8}k$, and we constructed degree-partitioned minimally $k$-(edge-)connected graphs which attain this upper bound asymptotically.

\begin{appendices}
\section{}\label{cases345}

\noindent{\em Proof of  Theorem~\ref{upperbound_sharp} for $k \in \{3,4,5\}$, i.e., 
if $k\in\{3,4,5\}$, then there is an infinite family of degree-partitioned minimally $k$-connected graphs whose average connectivity is asymptotically $\frac{9}{8}k$.}

Fix $k\in\{3,4,5\}$, and define $s=k^3-k^2$.  Let $p \geq 4s$, and define
\begin{align*}
W&=\{w_0,w_1,\ldots,w_{p-1}\},\\
X&=\{x_0,x_1,\ldots,x_{p-1}\},\\
Y&=\{y_0,y_1,\ldots,y_{p-1}\}, \text{ and}\\
Z&=\{z_0,z_1,\ldots,z_{p-1}\}.
\end{align*}
Let $\Psi_{k,p}$ be the graph of order $4p$ with vertex set $W\cup X\cup Y\cup Z$ and edge set $E_X\cup E_Y\cup E_Z$, where
\begin{align*}
E_X&=\{w_ix_{i+j}\ |\ 0 \le i \le p-1, 0\leq j\leq k-1\},\\
E_Y&=\{w_iy_{i+kj}\ |\ 0 \le i \le p-1, 0\leq j\leq k-1\}, \text{ and}\\
E_Z&=\{w_iz_{i+k^2j}\ |\ 0 \le i \le p-1, 0\leq j\leq k-1\},
\end{align*}
where subscripts are expressed modulo $p$.  


In a manner similar to the proof that $\Phi_{k,p}$ is degree-partitioned minimally $k$-connected, it can be shown that this also holds for $\Psi_{k,p}$.

We now compute $\kappa(u,v)$ for every pair $u,v$ of distinct vertices of $\Psi_{k,p}$.  First of all, if $u$ and $v$ are not both in $W$, then $\kappa(u,v)=k$, since $\Psi_{k,p}$ is $k$-connected and at least one of $u$ or $v$ must have degree $k$.  So it remains to consider the case that $u,v\in W$.

\medskip

\noindent
\textbf{Claim:} If $u,v\in W$, then $\kappa(u,v)=3k$.

\begin{proof}[Proof of Claim]
By relabelling vertices if necessary, we may assume that $u=w_0$ and $v=w_t$, where $1\leq t\leq \frac{p}{2}$.  For every pair of integers $i$ and $j$ with $i\leq j$, we define
\begin{align*}
W[i,j]&=\{w_i,w_{i+1},\ldots,w_j\},\\
X[i,j]&=\{x_i,x_{i+1},\ldots,x_j\},\\
Y[i,j]&=\{y_i,y_{i+1},\ldots,y_j\}, \text{ and}\\
Z[i,j]&=\{z_i,z_{i+1},\ldots,z_j\}, 
\end{align*}
where the subscripts are taken modulo $p$.  We also define $\Psi_{k,p}[i,j]$ to be the subgraph of $\Psi_{k,p}$ induced by the set $W[i,j]\cup X[i,j]\cup Y[i,j]\cup Z[i,j]$.  

Recall that $s=k^3-k^2$.  We will show that there is a collection of $3k$ internally disjoint paths from $w_0$ to $w_t$, all of whose vertices belong to the subgraph $\Psi_{k,p}[-k^2,t+s]$.  (Since $p\geq 4s$ and $t\leq \frac{p}{2}$, this is a proper subgraph of $\Psi_{k,p}$.)

First of all, if $t< 2s$, then we verify by computer that such a collection of $3k$ internally disjoint $w_0$--$w_t$ paths exists.\footnote{A Sage Jupyter Notebook containing the code used to perform this check, and the other checks in the remainder of this proof, can be found at \href{http://ion.uwinnipeg.ca/~lmol/Research}{http://ion.uwinnipeg.ca/$\sim$lmol/Research}.}   (Since $k\in \{3,4,5\}$ and $1\leq t<2s$, there are only finitely many cases to check.)   

\begin{figure}
\begin{center}
\begin{tikzpicture}[scale=0.8]
\useasboundingbox (-2.5,-0.5) rectangle (16.6,2.55);
\foreach \x in {0,1,3,4,10,11}
{
\vertex (\x) at (\x,0) {};
\pgfmathtruncatemacro{\y}{\x+3}
\vertex (z\y) at (\y,2) {};
\path (\x) edge (z\y);
}
\foreach \x in {6,7}
{
\vertex (\x) at (\x,0) {};
\pgfmathtruncatemacro{\y}{\x+4}
\vertex (z\y) at (\y,2) {};
}
\foreach \x in {3,4,6,7,10,11}
{
\path (\x) edge (z\x);
}
\foreach \x in {0.55,3.55,6.55,10.55}
{
\node at (\x,0) {\scriptsize $\cdots$};
}
\foreach \x in {3.55,6.55,10.55,13.55}
{
\node at (\x,2) {\scriptsize $\cdots$};
}
\foreach \x in {2.1,3.55,5.1,6.55,10.55,12.1}
{
\node at (\x,1) {\scriptsize $\cdots$};
}
\foreach \x in {6,7}
{
\draw (\x) -- (\x+0.75,0.5);
}
\foreach \x in {10,11}
{
\draw (z\x) -- (\x-0.75,1.5);
}
\node[below] at (0) {\tiny $w_{r+1}$};
\node[below] at (1) {\tiny $w_{r+3k}$};
\node[below] at (3) {\tiny $w_{r+s+1}$};
\node[below,xshift=8pt] at (4) {\tiny $w_{r+s+3k}$};
\node[below] at (6) {\tiny $w_{r+2p+1}$};
\node[below,xshift=10pt] at (7) {\tiny $w_{r+2s+3k}$};
\node[below,xshift=-5pt] at (10) {\tiny $w_{t-2s+1}$};
\node[below,xshift=5pt] at (11) {\tiny $w_{t-2s+3k}$};
\node[above] at (z3) {\tiny $z_{r+s+1}$};
\node[above,xshift=8pt] at (z4) {\tiny $z_{r+s+3k}$};
\node[above] at (z6) {\tiny $z_{r+2p+1}$};
\node[above,xshift=10pt] at (z7) {\tiny $z_{r+2s+3k}$};
\node[above,xshift=-5pt] at (z10) {\tiny $z_{t-2s+1}$};
\node[above,xshift=5pt] at (z11) {\tiny $z_{t-2s+3k}$};
\node[above] at (z13) {\tiny $z_{t-s+1}$};
\node[above,xshift=5pt] at (z14) {\tiny $z_{t-s+3k}$};
\vertex (w0) at (-1.5,0) {};
\node[below] at (w0) {\tiny $w_0$};
\vertex (wt) at (15.5,0) {};
\node[below] at (wt) {\tiny $w_t$};
\path[dashed] (w0) edge[bend left=90,looseness=2] (0);
\path[dashed] (w0) edge[out=135,in=45,looseness=4] (1);
\path[dashed] (z13) edge[out=-90,in=180] (wt);
\path[dashed] (z14) edge[out=0,in=30,looseness=2] (wt);
\node at (0,1.5) {$\mathcal{P}$};
\node at (-0.75,1.75) {$\vdots$};
\node at (14.6,1.2) {$\mathcal{R}$};
\node at (15.2,0.5) {$\cdots$};
\node at (8.55,1) {$\cdots$};
\node at (8.5,1.5) {$\mathcal{Q}$};
\draw[dotted] (-2.5,-0.5) rectangle (2.3,2.55);
\draw[dotted] (12.4,-0.5) rectangle (16.6,2.55);
\draw[dotted] (2.3,-0.5) -- (12.4,-0.5);
\draw[dotted] (2.3,2.55) -- (12.4,2.55);
\end{tikzpicture}
\end{center}
\caption{A collection of $3k$ internally disjoint $w_0$--$w_t$ paths in $\Psi_{k,p}[-k^2,t+s]$ for $t\geq 2s$, where $r= t\bmod s$.  The dotted box on the left contains the subgraph $\Psi_{k,p}[-k^2,r+s]$, the dotted box in the middle contains the subgraph $\Psi_{k,p}[r+s+1,t-s]$, and the dotted box on the right contains the subgraph $\Psi_{k,p}[t-s+1,t+s]$.}
\label{Fig:ThreeSegments}
\end{figure}
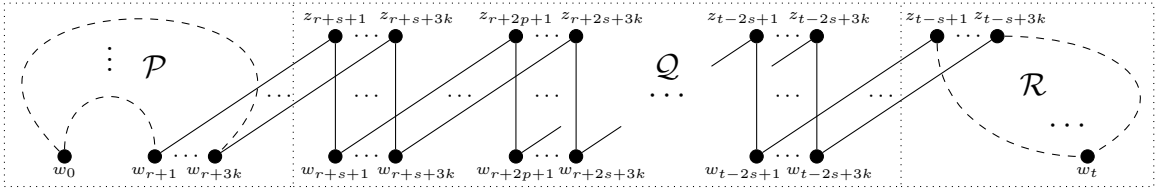

So we may assume that $t\geq 2s$.  Let $r= t\bmod s$.  We describe the collection of $3k$ internally disjoint $w_0$--$w_t$ paths in three separate segments -- see Figure~\ref{Fig:ThreeSegments}.
\begin{itemize}
\item We verify by computer that there is a collection $\mathcal{P}$ of $3k$ internally disjoint paths $P_1,\ldots, P_{3k}$ originating from $w_0$ and ending in the vertices $w_{r+1},\ldots,w_{r+3k}$, respectively, all of whose vertices belong to the subgraph $\Psi_{k,p}[-k^2,r+s]$.  (Again, since $k\in\{3,4,5\}$ and $0\leq r<s$, there are only finitely many cases to check.)
\item By definition of $\Psi_{k,p}$, the vertex $w_i$ is adjacent to both $z_i$ and $z_{i+s}$ for all $i$. We define a collection of paths by
\[
\mathcal{Q}=\{w_{r+j}z_{r+s+j}w_{r+s+j}z_{r+2s+j}...w_{t-2s+j}z_{t-s+j}\ |\ 1\leq j\leq 3k\}.
\]
(We use the fact that $r=t\bmod s$ here.)  Observe that $\mathcal{Q}$ is a collection of $3k$ internally disjoint paths $Q_1,\ldots, Q_{3k}$ originating from the vertices $w_{r+1},\ldots,w_{r+3k}$, respectively, and ending in the vertices $z_{t-s+1},\ldots,z_{t-s+3k}$, respectively, all of whose internal vertices belong to the subgraph $\Psi_{k,p}[r+s+1,t-s+3k]$.
\item We verify by computer that there is a collection $\mathcal{R}$ of $3k$ internally disjoint paths $R_1,\ldots,R_{3k}$ originating from the vertices $z_{t-s+1},\ldots, z_{t-s+3k}$, respectively, and ending in the vertex $w_t$, all of whose vertices belong to the subgraph $\Psi_{k,p}[t-s+1,t+s]$.
\end{itemize}
By concatenating the paths in $\mathcal{P}$, $\mathcal{Q}$, and $\mathcal{R}$ in the obvious manner, we obtain a collection of $3k$ internally disjoint $w_0$--$w_t$ paths in $\Psi_{k,p}[-k^2,t+s]$. 
\end{proof}
We conclude that $\Psi_{k,p}$ is degree-partitioned minimally $k$-connected, and that the average connectivity of $\Psi_{k,p}$ is given by $$\frac{3k\binom{p}{2} + k [\binom{4p}{2} -\binom{p}{2}]}{\binom{4p}{2}}=\left(\frac{9p-3}{8p-2}\right)k,$$ which is asymptotically $\frac{9}{8}k$.


\end{appendices}



\begin{thebibliography}{99}
\bibitem{Abajo2013}
E. Abajo, R.~M. Casablanca, A. Di{\'{a}}nez, and P. Garc{\'{i}}a-V{\'{a}}zquez,
  \emph{{On average connectivity of the strong product of graphs}}, Discrete
  Appl. Math. \textbf{161}({18}) (2013), 2795--2801.

\bibitem{BeinekeOellermannPippert2002}
L.~W. Beineke, O.~R. Oellermann, and R.~E. Pippert, \emph{{The average
  connectivity of a graph}}, Discrete Math. \textbf{252}({1-3}) (2002), 31--45.
  
  \bibitem{Boeing2017}
G. Boeing, \emph{{OSMnx: New methods for acquiring, constructing, analyzing,
  and visualizing complex street networks}}, Computers, Environment and Urban
  Systems \textbf{65} (2017), 126--139.


\bibitem{CasablancaMolOellermann2018}
R. Casablanca, L. Mol, and O.R. Oellermann,  \emph{Average connectivity of minimally $2$-connected graphs and average edge-connectivity of minimally $2$-edge-connected graphs}, Discrete Appl. Math. \textbf{289} (2021), 233--247.

\bibitem{DankelmannOellermann2003}
P. Dankelmann and O.~R. Oellermann, \emph{{Bounds on the average connectivity
  of a graph}}, Discrete Appl. Math. \textbf{129}({2-3}) (2003), 305--318.


\bibitem{DankelmannOellermann2005}
P. Dankelmann and O.~R. Oellermann, \emph{{Degree sequences of optimally
  edge-connected multigraphs}}, Ars Combin. \textbf{77} (2005), 161--168.
  
  \bibitem{KimO2013}
J. Kim and S. O, \emph{{Average connectivity and average edge-connectivity in
  graphs}}, Discrete Math. \textbf{313}({20}) (2013), 2232--2238.
  
  \bibitem{LiYang2012} H. Li and W. Yang, \emph{{Every $3$-connected essentially $10$-connected line graph is
Hamilton-connected}}, Discrete Mathematics \textbf{312} (2012), 3670--3674
  
  \bibitem{Mader1972}
W. Mader, \emph{{Ecken vom grad $n$ in minimalen $n$-fach zusammenh\"{a}ngenden
  graphen}}, Arch. Math. \textbf{23} (1972), 219--224.

\bibitem{Menger1927}
K. Menger, \emph{{Zur allgemeinen Kurventheorie}}, Fund. Math. \textbf{10}
  (1927), 96--115.
  
  \bibitem{Rak2015}
J. Rak, M. Pickavet, K.~S. Trivedi, J.~A. Lopez, A.~M. Koster, J.~P. Sterbenz,
  E.~K. {\c{C}}etinkaya, T. Gomes, M. Gunkel, K. Walkowiak, and D. Staessens,
  \emph{{Future research directions in design of reliable communication
  systems}}, Telecommunication Systems \textbf{60}({4}) (2015), 423--450.


\bibitem{Whitney1932}
H. Whitney, \emph{{Congruent graphs and the connectivity of graphs}}, Amer. J.
  Math. \textbf{54}({1}) (1932), 150--168.

  
\end{thebibliography}
\end{document}